\documentclass[12pt]{amsart}
\usepackage[utf8]{inputenc}

\usepackage{graphicx}

\usepackage{amsmath, amssymb, amsthm, amsbsy,  tikz,array}
\usetikzlibrary{scopes, backgrounds}
\usepackage{amsfonts}
\usepackage{fullpage}
\usepackage{comment}
\usepackage{dynkin-diagrams}

\newcommand{\ZZ}{{\mathbb{Z}}}

\newcommand\zzn{{\ZZ_3^n \times \ZZ}}

\DeclareMathOperator\Ker{\rm Ker}
\DeclareMathOperator\Irr{\rm Irr}

\DeclareMathOperator\tC{\widetilde{C}}
\DeclareMathOperator\tA{\widetilde{A}}
\DeclareMathOperator\tB{\widetilde{B}}

\DeclareMathOperator\tX{\widetilde{X}}

\newtheorem{theorem}{Theorem}[section]
\newtheorem{hypothesis}[theorem]{Hypothesis}
\newtheorem{proposition}[theorem]{Proposition}

\newtheorem{corollary}[theorem]{Corollary}

\newenvironment{proofof}[1]{\noindent{\it Proof of #1.}}
{\hfill$\square$\\\mbox{}}

\newtheorem{question}[theorem]{Question}

\newtheorem{remark}[theorem]{Remark}

\newcommand\Tstrut{\rule{0pt}{2.9ex}}         
\newcommand\Bstrut{\rule[-1.2ex]{0pt}{0pt}}   
\newcommand\TBstrut{\Tstrut\Bstrut}           

\title{Gelfand pairs for affine Weyl groups}

\author{P\'al Heged\"us}
\address{Budapest University of Technology and Economics, Institute of Mathematics, Műegyetem rkp 3, H-1111, Budapest, Hungary}
\email{hegpal@math.bme.hu}

\date{
	\today}

\thanks{
This research was partially supported by Hungarian National Research, Development and Innovation Office (NKFIH) Grant No.~K138596. The project leading to this application has received funding from the European Research Council (ERC) under the European Union's Horizon 2020 research and innovation programme (grant agreement No. 741420).
}

\begin{document}
	\begin{abstract} This paper is motivated by several combinatorial actions of the affine Weyl group of type $\tC_n$.
		Addressing a question of David Vogan, 
		it was shown in an earlier paper that 
		these permutation representations are essentialy multiplicity-free~\cite{AHR}.   
		However, the Gelfand trick, which was indispensable in~\cite{AHR} to prove this property for types $\tC_n$ and $\tB_n$, is not applicable for other classical types.
		Here we present a unified approach to  
		fully answer the analogous question for all irreducible affine Weyl groups.
		Given a finite Weyl group 
		$W$ with maximal parabolic subgroup $P\leq W$, there corresponds to it a reflection subgroup $H$ of the affine Weyl group $\widetilde W$.
		It turns out that while the Gelfand property of $P\leq W$ does not imply that of $H\leq \widetilde{W}$, but $Q=N_W(P)\leq W$ has the Gelfand property if and only if $K=QH\leq \widetilde{W}$ has.
		Finally, for each irreducible type we describe when $(W,Q)$ is a Gelfand pair.
	\end{abstract}
	\maketitle
	
	\section{Introduction}
	Our previous paper~\cite{AHR} was motivated by several natural finite actions of the affine Weyl group $G$ of type $\tC_n$, see Section~2 there. We concluded that those were ``essentially'' multiplicity-free, that is, the actions are not multiplicity-free, but there exist a $G$-invariant coupling of the underlying set such that the $G$-action on the couples is multiplicity-free.
	
	Our approach included an inflated action of $G$ on an infinite set $\zzn$ with point stabilizer $H$. Dealing with infinite action led to the definition of the \emph{proto-Gelfand pair}\cite[Definition~3.3]{AHR}: let $A\leq B$ for a not necessarily finite group $B$, we call $(B,A)$ a proto-Gelfand pair if for every homomorphism $\varphi:B\mapsto B_1$ with finite image $(\varphi(B),\varphi(A))$ is a Gelfand pair. We also say that $A$ is a proto-Gelfand subgroup of $B$.
	
	We proved
	\begin{theorem}\cite[Theorem~1.7]{AHR}
		Let $G$ be an affine Weyl group of type $\tC_n$ or $\tB_n$.
		For every $\omega\in \zzn$ whose stabilizer is not all of $G$ (namely, $\omega$ is not $G$-invariant), there exists a double cover of the stabilizer which is a proto-Gelfand subgroup of $G$.
	\end{theorem}
	
	The action above can be understood in terms of the affine permutations and also in terms of the fundamental domain.
	Here we do not explore the possible actions but concentrate only on the Gelfand property of the stabilizer. We prove the following, analogous theorem for all the affine Weyl groups. As a general reference for Coxeter groups in general and affine Weyl groups in particular we use \cite{Humphreys:1990}.

	Let $\Phi$ be an irreducible root system 
	with fundamental roots $\Delta$. Let $W$ be its Weyl group, $X=L(\Phi^\vee)$ the translation group with its coroot lattice and $G=W\ltimes X$ the corresponding affine Weyl group.
	
	Delete a node from $\Delta$ and let $\Phi_0$ be the (typically reducible) root system generated by the remaining fundamental roots 
	and let $H$ be the reflection subgroup of $G$ generated by the affine reflections corresponding to $\Phi_0$.
	
	\begin{theorem}\label{main}
		If $P=H\cap W$ or its normalizer $Q=N_W(P)$ is a Gelfand subgroup of $W$ then $K=QH$ is a proto-Gelfand subgroup of $G$. Otherwise it is not.
		
		In particular, this property holds for types
		\begin{description}
			\item[$A_n,B_n,C_n,G_2$] for every removed node;			
			\item[$D_n$] for every removed node if $n$ is odd and for $k\leq 2[n/4]+1$ or $k\geq n-1$ if $n$ is even (for the numbering see Figure~\ref{fig:rootsys}, below);
			\item[$E_6,E_7$] if one of the three degree one nodes is removed;
			\item[$E_8$] if one of the two farthest endnodes is removed;
			\item[$F_4$] if one of the two endnodes is removed.
		\end{description}
	\end{theorem}
	The converse part of the theorem is easy. If $K$ is a proto-Gelfand subgroup of $G$ then so its overgroup $KX$, which is of finite index in $G$ is a Gelfand subgroup of $G$. But $KX$ contains the normal subgroup $X$, so its being a Gelfand subgroup of $G$ is equivalent to $Q\cong KX/X$ being a Gelfand subgroup of $W\cong G/X$.
	
	If $P$ is not a Gelfand subgroup of $W$ then, by a similar argument, $H$ cannot be a proto-Gelfand subgroup of $G$. However, even if $P$ is Gelfand in $W$ but $P<Q$ then $H$ is still not a proto-Gelfand subgroup of $G$, that is why $K$ is needed in the theorem instead of the reflection subgroup $H$. See Remark~\ref{rem:nonGelfand}, below.
	
	Our proof in \cite{AHR} was based on the so called ``Gelfand's trick:'' if $G$ possesses an involutive anti-automorphism (typically the group inversion) that fixes every $H-H$ double coset then $(G,H)$ is a (proto-)Gelfand pair. In the present paper  
	we use character theory for two reasons. First, that it enables a unified approach of all affine Weyl groups. Second, that in some cases Gelfand's trick does not suffice. One such is the type $\tA_n$. The following is a special case of our main theorem, removing the $k$-th node from the Dynkin diagram of $A_n$.
	\begin{corollary}\label{th:A_n}Let $G$ be a Weyl group of type $\tA_n$. It contains naturally a reflection subgroup $H_k$ of type $\tA_{k-1}\times\tA_{n-k}$. If $2k\ne n+1$ then $H_k$ is a proto-Gelfand subgroup of $G$. If $2k=n+1$ then $H_k$ has a double cover $K_k$ which is a proto-Gelfand subgroup of $G$.
	\end{corollary}
	It is possible to show that the double coset algebra, the Hecke algebra, $H(G,H_k)$ (or $H(G,K_k)$ for $2k=n+1$) is commutative, so it is for every finite quotient. Still, the simplest Gelfand's trick does not work, not even for $2k\ne n+1$, the double coset $H_kxH_k=xH_k\ne H_k$ (with $x\in X$)  does not contain involutions. While $H_k$ is not a parabolic subgroup, this phenomenon also indicates why in the theorem of Curtis, Iwahori and Kilmoyer (see Theorem~\ref{thm:CIK}, below) finiteness is required.
	
	The study of the Gelfand property in Coxeter groups goes back to at least half a century. To our knowledge 
	\cite{CIK} and \cite{saxl} are among the first general results. 
	Saxl \cite{saxl} gave a list of potential candidates of Gelfand subgroups $G\leq S_n$ for $n>18$. His list was later made exact by Godsil and Meagher \cite{GM} where they dealt with $n\leq 18$, too.
	
	Curtis, Iwahori and Kilmoyer considered parabolic subgroups of finite Coxeter groups. The following is abridged from \cite[Theorem~3.1]{CIK}.
	\begin{theorem}\label{thm:CIK}
		Let $(W,R)$ be a Coxeter system with $W$ finite, $J\subseteq R$ and $W_J=\langle s_j|j\in J\rangle$ the corresponding parabolic subgroup. The Hecke algebra (double coset algebra) $H(W,W_J)$ over an algebraically closed field of characteristic $0$ is commutative if and only if the shortest element of each double coset $W_J w W_J\subseteq W$ is an involution.
	\end{theorem}
	
	Complete classification of commutative Hecke algebras of Coxeter groups over parabolic subgroups is given by Abramenko, Parkinson and Van Maldeghen\cite{APVM}. Let $X_{n,i}$ denote the case of the Coxeter system $(W,S)$ of type $X_n$ and $I=S\setminus\{i\}$, removing node $i$ from $S$ according to the standard numbering of the nodes. Similarly $\tX_{ n , i}$ denotes the case when $I=S\setminus\{i\}$ for the affine Weyl group of type $\tX_n$.
	\begin{theorem}\cite[Theorem~2.1]{APVM}\label{thm:APVM}
		Let $( W , S )$ be an irreducible Coxeter system,  $I \subseteq S$ be spherical (that is, $W_I$ is finite), and let specialization $\mathbf{\tau} = ( \tau_s )$ with $\tau_s\geq 1$ for each $s\in S$. The
		$I$-parabolic Hecke algebra $\mathcal{H}^I$ and its specialization $\mathcal{H}^I_\tau$ are noncommutative if $| S \setminus I | >1$. If $I = S \setminus\{ i \}$ then $\mathcal{H}^I$ and $\mathcal{H}^I_\tau$  are
		commutative in the cases
		\begin{itemize}
			\item $A_{n , i}\, ( 1 \leq i \leq n )$ , $B_{n , i}\, ( 1 \leq i \leq n )$ , $D_{n , i}\, ( 1 \leq i \leq n/2\text{ or } i= n - 1 , n)$, $E_{6 , 1}$, $E_{6 , 2}$, $E_{6 , 6}$, $E_{7 , 1}$, $E_{7 , 2}$, $E_{7 , 7}$, $E_{8 , 1}$, $E_{8 , 8}$, $F_{4 , 1}$, $F_{4 , 4}$, $H_{3 , 1}$, $H_{3 , 3}$, $H_{4 , 1}$, $I_{2}(p)_{i} (i=1,2)$, and\\
			\item all affine cases $\tX_{ n , i}$ with $i$ a special type (that is, the removed node and the extra node of the Coxeter diagram are in the same orbit under graph automorphisms),
		\end{itemize}
		and noncommutative otherwise.
	\end{theorem}
	We will use the list of this theorem for exceptional Weyl groups in the following reference form. The parabolic subgroup is a Gelfand subgroup if and only if the deleted node is an endnode (``leaf'') of the Dynkin diagram but not the middle leaf of $E_8$. See Figure~\ref{fig:rootsys} for the labelling of the diagrams. 
	
	We finish the Introduction by a natural question. Beyond parabolic subgroups one may 
	consider all reflection subgroups. 
	Dyer and Lehrer gave a complete description of reflection subgroups of finite Coxeter groups and affine Weyl groups\cite{DL}. Our subgroup $H$ is always such. The following is a slight modification of Question~6.2 of \cite{AHR}.
	\begin{question}
		Which reflection subgroups of affine Weyl groups (or finite covers of them) are proto-Gelfand?
	\end{question}
	\noindent
	{\bf Acknowledgement.} This paper is an outgrowth from a joint work with Ron Adin and Yuval Roichman\cite{AHR}, but uses quite different methods. Beyond friendship and hospitality, they helped in the development of the present work in several ways both in structure and in content. The author wishes to express gratitude towards them.
	
	\section{Gelfand pairs for affine Weyl groups}

	We prove first a general result and then show that its assumptions hold for the affine Weyl groups. These assumptions are collected in the following hypothesis about a group $G$. Note that the definition of $Q$ in (e) is different from the definition of $Q$ in Theorem~\ref{main}. That they are the same will become clear only in the proof of the theorem. When we establish property (c) we also prove $N_W(Y)=N_W(P)$, see there. For the proof of Proposition~\ref{prop:main} the version below is more convenient.
	
	\begin{hypothesis}\label{hypo}
		Let $G=W\ltimes X$ be a semidirect product of a finite group $W$ and a free Abelian group $X$.
		Further, let $H\leq G$ be a subgroup for which the following (a), (b), (c), (d), (e) or (a), (b), (c'), (d'), (e')
		hold.
		\begin{enumerate}
			\item $X\cap H=Y$ and $X=Y\times\langle x\rangle$ for some $x\in X$, which we fix;
			\item $H\triangleleft L=H\langle x\rangle$;
			\item $|N_G(Y):L|\leq 2$;
			\item if $g\in G\setminus N_G(Y)$ then there exist $y,z\in Y$ such that $y^g=x^m z$ with $m\in\{\pm1,\pm2\}$;
			\item $Q=N_W(Y)$ is a Gelfand subgroup of $W$.
			\item[(c')] $|N_G(Y):L|=2$;
			\item[(d')] if $g\in G\setminus N_G(Y)$ then there exist $y,z\in Y$ such that $y^g=x^m z$ with $m\in\{\pm1,\pm2,\pm3\}$;
			\item[(e')] $P=H\cap W$ is a Gelfand subgroup of $W$.
		\end{enumerate}
	\end{hypothesis}
	
	\begin{figure*}[h]
		\centering
		
		\begin{tikzpicture}[scale=.6]
		\node (G) at (1,4) {$G$};
		\node (W) at (-3,2) {$W$};
		\node (KL) at (2,2) {$KL$};
		\node (K) at (0,1) {$K$};
		\node (L) at (3,0) {$L$};
		\node (K0) at (-2,0) {$Q$};
		\node (X) at (4,-2) {$X$};
		\node (H) at (1,-1) {$H$};
		\node (X1) at (2,-3) {$Y$};
		\node (H0) at (-1,-2) {$P$};
		\node (one) at (0,-4) {$1$};
		\draw (G) -- (W) -- (K0) -- (H0) -- (one) -- (X1) -- (X) -- (L) -- (KL) -- (G);
		\draw (KL) -- (K) -- (K0);
		\draw (L) -- (H) -- (H0);
		\draw (K) -- (H) -- (X1);
		\end{tikzpicture}
		\caption{Part of the subgroup lattice of $G$}
		\label{fig:subgrouplattice}
	\end{figure*}
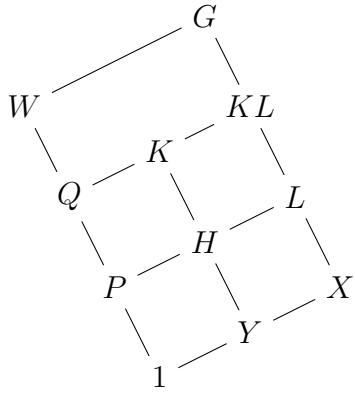
	\begin{remark}\label{prop:g2}The (c'), (d'), (e') version is needed in the sole case when $\Phi$ is a root system of type $G_2$ when we remove the short fundamental root. Here (d) does not hold so we somewhat relax it and strengthen (c) and (e).
		
		After the proof of the main case of Proposition~\ref{prop:main} we briefly cover the necessary changes.
	\end{remark}
	\begin{proposition}\label{prop:main}
		With the assumptions of Hypothesis~\ref{hypo}, $K=N_W(Y)Y$
		is a proto-Gelfand subgroup of $G$.
	\end{proposition}
	\begin{proof}
		Observe, that $N_G(Y)=KX=KL\triangleright H$ with $L/H\triangleleft KL/H$ cyclic. If $K/H$ is also normal in $L/H$ then 
		all the assumptions hold for $K$ in place of $H$. So, in the following we assume that either $K=H$ or $K>H$ and not normal in $L$.
		
		It also follows from the above properties 
		that $N_{G}(Y)$ acts on the infinite cyclic group $X/Y$ so for any $g\in N_{G}(Y)$ we have $x^g\in xY$ or $x^g\in x^{-1}Y$. Note also that, by (b), $H$ trivially acts on $X/Y$.
		
		We need show that the permutation character is multiplicity-free in every finite quotient. Let $N\triangleleft G$ be any normal subgroup of finite index. It contains $N\cap X$, still of finite index in $X$. Given this finite index we can pick $h$ sufficiently large so that $N\cap X\supseteq X^{(h)}=\{y^h\mid y\in X\}$ which is still a normal subgroup of finite index in $G$. If $ZX^{(h)}/X^{(h)}$ is a Gelfand subgroup of $G/X^{(h)}$ then all the more so is $ZN/N$ a Gelfand subgroup of $G/N$. So it is enough to consider factors by the normal subgroups $X^{(h)}$. For ease of notation we use the same symbol for subgroups of the factor group as for subgroups of $G$. In other words, instead of being a free Ableian group, $X$ is assumed to be isomorphic to the homogeneous Abelian group $Z_h^n$. The rest of Hypothesis~\ref{hypo} remains intact.
		
		In particular, $L/H$ is cyclic or order $h$ and if $K>H$ then $KL/H$ is dihedral of order $2h$. As $1_K^{G}=(1_K^{KL})^G$, first we treat the irreducible constituents of $1_K^{KL}$ separately.
		
		{\bf Claim1:}
		Suppose $\varphi\in\Irr(KL)$ is a linear constituent of $1_K^{KL}$ that is $\varphi_{\downarrow K}=1_K$. 
		Then $\varphi^G$ is multiplicity free.
		
		Indeed, as $G=WKL$, we may use a special case of Mackey's theorem \cite[Problem~(5.2)]{Isaacs:1976} for the restriction of the induced character ${\varphi^G}_{\downarrow W}=(\varphi_{\downarrow K\cap W})^W=1_{K\cap W}^W$ which is multiplicity free by (e). Hence $\varphi^G$ itself is also multiplicity free. 
		
		Let $\varepsilon$ be a primitive $h$-th root of $1$ and $\eta$ a linear character of $L$ with kernel $H$ and $\eta(x)=\varepsilon$. 
		We consider $1_{H}^{L}=\sum_{s=1}^{h}\eta^s$, a sum of linear characters, each a power of $\eta$. For any integer $i$ and $g\in H$ we have $\eta^t(gx^i)=\varepsilon^{ti}$.	
		Let $\lambda_t=\eta^t_{\downarrow X}\in\Irr(X)$ for $t=1,\ldots,h$. 
		
		{\bf Claim2:} 
		If for some $g\in G$, $\lambda_t=\lambda_s^g$, where $t\ne h,h/2$ then $g\in KL$, so $s=t$ or $s=h-t$ and $K>H$.
		
		Indeed, if $g\notin KL=N_G(Y)$ then, by assumption (d), there exists $y,z\in Y$ such that $y^g=x^mz$ with $m\in\{\pm1,\pm2\}$. So $1=\lambda_s(y)=\lambda_s^g(x^m z)=\lambda_t(x^m)=\varepsilon^{mt}$, impossible.

		{\bf Claim3:}
		If $(\eta^t)^G\ne(\eta^s)^G$ then they share no common irreducible constituent. Of course, if $K>H$ and $s=h-t$ then $(\eta^s)^{KL}=(\eta^t)^{KL}$ so $(\eta^t)^G=(\eta^s)^G$. 
		
		Note that
		the irreducible constituents of $(\eta^t)^G$ lie above $\lambda_t$. So, by Clifford's Theorem\cite[(6.2)]{Isaacs:1976}, it is enough to prove that $\lambda_t$ and $\lambda_s$ are not $G$-conjugate, unless $s=t$ or $s=h-t$ and $K>H$. By {Claim2}, we have to check only $s,t\in\{h,h/2\}$. But then $s=t$ as $\Ker(\lambda_h)=G\ne \Ker(\lambda_{h/2})$.

		If $K=H$ then, by Claim1, all $(\eta^s)^G$ are multiplicity free and, by Claim3, the $(\eta^s)^G$ ($s=1,\ldots,h$) share no common irreducible constituent. Hence $1_H^G=\sum_{s=1}^h (\eta^s)^G$ is multiplicity free, indeed.
		
		So let $K>H$. If $t=h,h/2$ (the second only for even $h$) then 
		there are unique extensions $1_{KL}$ of $\eta^h=1_L$ and $\mu$ of $\eta^{h/2}$ (for $h$ even) to $KL$ that are trivial on $K$. So, using Claim1 again, we conclude that 
		$1_{KL}^G$ and $\mu^G$ are multiplicity free with all constituents lying above $\lambda_h=1_X$ and $\lambda_{h/2}$, respectively.
		
		Assume now $t\ne h,h/2$. Let $I_t=I_G(\lambda_t)=\{g\in G\mid \lambda_t^g=\lambda_t\}\geq L$ denote the inertia subgroup of $\lambda_t$. By Claim2, 
		$I_t\subseteq KL$. If $g\in K\setminus H$ then $\lambda_t^g=\lambda_{h-t}\ne\lambda_t$ so $I_t=L$ and $(\eta^s)^G$ is irreducible as it is induced from the inertia subgroup\cite[(6.11)]{Isaacs:1976}.
		
		Now $1_K^{KL}$ is a sum of irreducible, degree $2$ characters $\sum_{s=1}^{\lfloor \frac{h-1}{2}\rfloor}(\eta^s)^{KL}$, of $1_{KL}$ and if $h$ is even then of an additional linear character $\mu$ extending $\eta^{h/2}$. For $1\leq s<h/2$ $(\eta^s)^G$ are distinct irreducible while 	$\mu^G$ (for $h$ even) 
		and $1_{KL}^G$ are multiplicity free. By Claim3, there are no shared constituents among these, so $1_K^G$ is multiplicity free.
		\bigskip
		
		If together with (a) and (b), instead of (c), (d) and (e), the alternatives (c'), (d'), (e') hold then we make the following modifications. For Claim1: if $\varphi\in\Irr(L)$ is a linear constituent of $1_H^{L}$ then $\varphi^G$ is multiplicity free. The proof is the same, using (e'). For Claim2, the above proof still works, using (d'), unless $a=\pm3$ and $t=h/3,2h/3$. But even in that case $s=t$ or $s=h-t$. So, if for some $g\in G,\ \lambda_t=\lambda_s^g$ then $s=t$ or $s=h-t$.
		
		Claim3 holds, that is $(\eta^t)^G,(\eta^s)^G$ share no common irreducible summand unless $s=t,s=h-t$ but in these cases they are the same, as $K>H$ by (c'). The conclusion of the proof is the same: we use Claim1 (and (c')) to show that for $t\ne h,h/2$ $(\eta^t)^G=((\eta^t)^{KL})^G$ are all multiplicity free. Hence, by Claim3 as above, $1_K^G$ is multiplicity free.
	\end{proof}
	\begin{remark}\label{rem:nonGelfand}
		If $H<K$ then $(G,H)$ is not a proto-Gelfand pair. Indeed, for the natural map $\varphi:G\rightarrow W\ltimes \ZZ_3^n$ the image $\varphi(KL)/\varphi(H)$ is a dihedral group of order $6$ so $(\varphi(KL),\varphi(H))$ is not a Gelfand pair. The proposition claims that even if $H$ itself is not a proto-Gelfand subgroup, a double cover is.
	\end{remark}
	\begin{proofof}{Theorem~\ref{main}}
		To finish the proof we have to confirm the assumptions of Hypothesis~\ref{hypo} for the subgroup $H$ defined before the statement of the theorem and described below in more detail. 
		
		Let $\Phi$ be a root system of 
		rank $n$, $W$ be its Weyl group and $G$ be the corresponding affine Weyl group. Let $\Delta=\{\alpha_i\}_{i=1,\ldots,n}$ denote the fundamental roots and  $\Phi^\vee=\{\alpha^\vee=\frac{2\alpha}{(\alpha,\alpha)}\mid \alpha\in\Phi\}$ the corresponding coroot system. Then the normal subgroup of translations $X$ is naturally identified with $L(\Phi^\vee)$, the coroot lattice. We record some properties of the root systems in Figure~\ref{fig:rootsys}, see \cite[Sections~2.9, 2.10, 4.9]{Humphreys:1990}.
		\begin{figure*}[h]\caption{Coefficients of the highest root in irreducible root sytems}\label{fig:rootsys}
			\[\begin{array}{r
				|l|l}
			&\text{Dynkin diagram}&\text{coefficients of }\tilde{\alpha}\Bstrut\\
			\hline\TBstrut
			A_n
			&\dynkin[labels={1,2,n-1,n},label macro/.code={\alpha_{\drlap#1}},edgelength=.85cm]A{}&1,\ldots,1\\
			\hline\TBstrut
			B_n
			&\dynkin[labels={1,2,n-2,n-1,n},label macro/.code={\alpha_{\drlap#1}},edgelength=.85cm]B{}&1,2,\ldots,2\\
			\hline\TBstrut
			C_n
			&\dynkin[labels={1,2,n-2,n-1,n},label macro/.code={\alpha_{\drlap#1}},edgelength=.85cm]C{}&2,\ldots,2,1\\
			\hline\TBstrut
			D_n
			&\dynkin[labels={1,2,n-3,n-2,n-1,n},label macro/.code={\alpha_{\drlap#1}},edgelength=.85cm,label directions={,,,right,,}]D{}&2,\ldots,2,1,1\\
			\hline\TBstrut
			E_6
			&\dynkin[labels={1,2,3,4,5,6},label macro/.code={\alpha_{\drlap#1}},edgelength=.85cm,label directions={,,,,,}]E6&1,2,2,3,2,1\\
			\hline\TBstrut
			E_7
			&\dynkin[labels={1,2,3,4,5,6,7},label macro/.code={\alpha_{\drlap#1}},edgelength=.85cm,label directions={,,,,,,}]E7&2,2,3,4,3,2,1\\
			\hline\TBstrut
			E_8
			&\dynkin[labels={1,2,3,4,5,6,7,8},label macro/.code={\alpha_{\drlap#1}},edgelength=.85cm,label directions={,,,,,,,}]E8&2,3,4,6,5,4,3,2\\
			\hline\TBstrut
			F_4
			&\dynkin[labels={1,2,3,4},label macro/.code={\alpha_{\drlap#1}},edgelength=.85cm]F4&2,3,4,2\\
			\hline\TBstrut
			G_2
			&\dynkin[reverse arrows,labels={1,2},label macro/.code={\alpha_{\drlap#1}},edgelength=.85cm]G2&3,2\\
			\end{array} 
			\]
		\end{figure*}
		Pick $\alpha$ from among the fundamental roots. Let $\Phi_0$ be the (reducible) root system generated by the remaining fundamental roots $\Delta\setminus\{\alpha\}$. Finally, let $H$ denote the reflection subgroup of $G$ generated by the affine reflections corresponding to $\Phi_0$. Note that $P=W\cap H$ is the parabolic subgroup $W_{\Delta\setminus\{\alpha\}}\leq W$ generated by the fundamental reflections corresponding to $\Delta\setminus\{\alpha\}$.

		Of course, $H$ is the direct product of the affine reflection groups on the (usually two) connected components of $\Delta\setminus\{\alpha\}$.
		Namely, if $\Gamma=\{\alpha_i\mid i\in I\}$ is one component then $H_\Gamma=\langle g,s_i\mid i\in I
		\rangle$ is one direct factor of $H$, where $g=s_{\beta,1}$ is an affine reflection flipping $\beta$, the highest root in $\Phi_\Gamma$, the root system generated by $\Gamma$. 
		Let $Y=H\cap X$, it is naturally isomorphic to $L(\Phi_0^\vee)\leq L(\Phi^\vee)$ and let $x=t(\alpha^\vee)$ the translation by the coroot in the direction of the missing fundamental root, $\alpha$. Of course, $X=Y\times\langle x\rangle$. So Property (a) follows.
		
		If $s=s_{\alpha_i}\in H$ is among the ordinary reflections generating $H$ then \[x^{-1}sxs=t(-\alpha^\vee)s_{\alpha_i}t(\alpha^\vee)s_{\alpha_i}=t(-\alpha^\vee)t(s_{\alpha_i}\alpha^\vee)=t(-\alpha^\vee)t(\alpha^\vee-(\alpha^\vee,\alpha_i)\alpha_i^\vee)\in Y.
		\]
		As all for the generators of $H/Y$ commute with $xY$, we get that $H/Y$ and $X/Y$ commute.  
		In other words, $H\triangleleft L=H\langle x\rangle$. So Property (b) follows.

		Recall, that $P=W\cap H$. 
		Let $U$ denote the intersection of the reflecting hyperplanes corresponding to the $n-1$ fundamental roots in $\Delta\setminus\{\alpha\}$. So $U=\langle\Phi_0\rangle^\perp$ is a line. Let $M=\{w\in W\mid wU=U\}=\{w\in W\mid w\Phi_0\subseteq\langle\Phi_0\rangle\}=N_W(Y)$. So the elements of $M$ either reflect $U$ through the origin or fix $U$ pointwise. In this latter case 
		$w$ is in the isotropy group of $U$, which is $P$, by \cite[Theorem~1.12(d)]{Humphreys:1990}. Hence $|M:P|\leq 2$, in particular $M\leq N_W(P)$. Conversely, it is clear that $U$ is  $N_W(P)$-invariant, so $N_W(P)\leq M$ and hence $N_W(Y)=M=N_W(P)$ and property (c) follows.

		To prove property (d) we will use that the coefficient of the deleted fundamental root $\alpha$ in the highest root $\tilde{\alpha}$ is $1$ or $2$. For types $A,B,C$ and $D$ the highest root has coefficients $1,2$ only while for the exceptional types this holds for the endnodes, the ``leaves.'' (See our remark after Theorem~\ref{thm:APVM} in the Introduction.) There are two exceptions: the middle leaf of $E_8$ (for which the parabolic subgroup itself is not a Gelfand subgroup) and the short fundamental root of $G_2$ (whose coefficient in the highest root is $3$). For the middle leaf of $E_8$ the parabolic subgroup is self-normalizing so (e) does not hold. So the theorem does not hold, either. For the short root of $G_2$ (c'), (d') and (e') do hold, see below.
		
		Suppose that $w\in W$ does not normalize $L(\Phi_0^\vee)$, hence there is a root $\beta\in\Phi_0$ such that in the decomposition of the root $w\beta\in\Phi$ the coefficient of $\alpha$ is non-zero. As both $\tilde{\alpha}-w\beta$ and $w\beta+\tilde{\alpha}$ are non-negative combinations of the fundamental roots so the coefficient of $\alpha$ in $w\beta$ is between $-2$ and $2$. As it is also non-zero so property (d) follows.
		
		For classical types property (e) follows from Theorem~\ref{thm:APVM} save the case of $D_{n,i}$ $n/2<i<n-1$, which requires special attention. To obtain $Q$ in each of the exceptional cases $E_6,\,E_7,\,E_8,\,F_4$ we used GAP\cite{GAP}. We conclude that if the node is not a leaf then $Q$ is not a Gelfand subgroup.

		\begin{description}
			\item[$A_n,B_n,C_n$] By Theorem~\ref{thm:APVM}, $P$ is a Gelfand subgroup, so (e) holds. 
			
			\item[$D_n$]  By Theorem~\ref{thm:APVM}, if we remove $\alpha_k$ then $P$ is a Gelfand subgroup, so (e) holds, unless $n/2<k<n-1$. 
			However, $Q=N_W(P)$ is 
			a Gelfand subgroup of $W$ if and only if $n$ is odd or $n$ is even and $k\leq 2[n/4]+1$ or $k\geq n-1$. See Theorem~\ref{thm:Dn} in the Appendix for a proof of these. Hence (e) holds unless $n$ is even and $2[n/4]+1<k<n-1$.

			\item[$E_6$] If $\alpha$ is one of the three endnodes then $P$ is a Gelfand subgroup, by Theorem~\ref{thm:APVM}, so (e) holds. For the other three nodes even $Q$ is not a Gelfand subgroup. (Among these, $Q>P$ holds only for the middle node.)

			\item[$E_7$] If $\alpha$ is and endnode then $P$ is a Gelfand subgroup, by Theorem~\ref{thm:APVM}, so (e) holds. For the other four nodes even $Q$ is not a Gelfand subgroup. (Even though $Q>P$ always.) 
			\item[$E_8$] If $\alpha$ is one of the two farthest endnodes then $P$ is a Gelfand subgroup, by Theorem~\ref{thm:APVM}, so (e) holds. For the other six nodes even $Q$ is not a Gelfand subgroup. (Even though $Q>P$ always.)
			\item[$F_4$] If $\alpha$ is an endnode then $P$ is a Gelfand subgroup, by Theorem~\ref{thm:APVM}, so (e) holds. For the two middle nodes even $Q$ is not a Gelfand subgroup. (Even though $Q>P$ always.)
			\item[$G_2$] The natural action of the dihedral group on the vertices of a $12$-gon is multiplicity-free, so (e') holds. As $Q=N_W(P)$ is of order $4$, a double cover, (c') holds. The coefficients of 
			$\tilde{\alpha}$ are $2$ and $3$, so (d') holds. 
		\end{description}
	\end{proofof}

	\section{Appendix}
	It seems that Lehrer\cite{lehrer} was the first to determine which parabolic subgroups of $D_n$ are Gelfand subgroups. His list was incomplete, later Abramenko, Parkinson and Van Maldeghem\cite{APVM} provided the complete answer.
	
	Here we prove the extension to the double cover but for completeness we also state the claim concerning the parabolic subgroups themselves. The proof resembles the proof of Theorem~\ref{main}.
	\begin{theorem}\label{thm:Dn}
		Let $n>3$ and $W$ a Weyl group of type $D_n$. Let $1\leq k\leq n$ and $P\le W$ the parabolic subgroup corresponding to removing the $k$-th node. It is a Gelfand subgroup if and only if $1\leq k\leq n/2$ or $k\geq n-1$. Suppose $n/2<k<n-1$. The double cover $Q=N_W(P)$ is a Gelfand subgroup of $W$ if and only if $n$ is odd or $n$ is divisible by $4$ and $k=n/2+1$.
	\end{theorem}
	\begin{proof}
		We use the isomorphism $W\cong V\rtimes S_n$, where $V=\mathbb{F}_2^{n-1}$. The action of $\sigma\in S_n$ on $V=\langle v_1,\ldots v_{n-1}\rangle$ is by 
		\[v_i\sigma=\begin{cases}
		v_{i\sigma},&\text{ if }n\sigma=n,\\
		v_{n\sigma},&\text{ if }i\sigma=n,\text{ while }\\
		v_{i\sigma}+v_{n\sigma},&\text{ if }i\sigma ,\,n\sigma<n.
		\end{cases}\]
		The claim about the Gelfand property of $P$ is covered by Theorem~\ref{thm:APVM}. From now on we assume $n/2<k<n-1$. Then $P$ is the Weyl group of the decomposable root system of type $A_{k-1}\oplus D_{n-k}$ (with minor notational changes due to degeneration if $k=n-3,n-2$) and $P\cong S_{k}\times(V_0\rtimes S_{n-k})$, where $V_0=\langle v_{k+1},\ldots v_{n-1}\rangle$. Let $Q=N_W(P)$. As $k>n/2$, $Q= N_V(P)P=\langle v_1+v_2+\cdots +v_k\rangle P$. So let $V_1=Q\cap V=\langle v_1+v_2+\cdots +v_k, v_{k+1},\ldots,v_{n-1}\rangle$.
		
		To determine $1_{Q}^W$ we first decompose $1_{Q}^{VQ}$. Using a special case of Mackey's theorem \cite[Problem~(5.2)]{Isaacs:1976}, ${1_{Q}^{VQ}}\downarrow_V=1_{V_1}^V$, whose constituents correspond to $E\subseteq\{1,\ldots, k\}$ of cardinality $|E|\equiv k\pmod{2}$. Namely, for each such $E$ let $\eta_E\in \Irr(V)$ be such that $v_j\in \Ker(\eta_E)$ if and only if $j\in E$ or $j>k$. Then $1_{V_1}^V=\sum_E\eta_E$. Note that the orbit of a subset $E\subseteq \{1,\ldots,k\}$ under the action of $S_k\times S_{n-k}$ consists of the subsets of the same cardinality. So $1_{Q}^{VQ}=\sum_{i=0}^{[k/2]}\chi_i$, where ${\chi_i}_V=\sum_{|E|=k-2i}\eta_E$. Fix $E=\{2i+1,\ldots,k\}$ and the inertia subgroup of $\eta_E$ in $VQ$ is $M=V(S_{2i}\times S_{k-2i}\times S_{n-k})=I_{VQ}(\eta_E)$. If $\nu_E\in\Irr(M)$ denotes the extension of $\eta_E$ to $M$ which is trivial on $S_{2i}\times S_{k-2i}\times S_{n-k}$, then $\chi_i=\nu_E^{VQ}$ is irreducible by\cite[(6.11)]{Isaacs:1976}.
		
		We claim that $\chi_i^W$ are all multiplicity free and $(\chi_i^W,\chi_j^W)>0$ if and only if $i=j$ or $2i+2j=n$. Justifying this claim is enough, as this latter possibility can occur only for even $n$. But if $n$ is divisible by $4$ then $k=n/2+1$ odd and the largest distinct $2i,2j$ are $k-1,k-3$, their sum is $2k-4=n-2$, so even this is impossible. (But for other even $n$ such pair of distinct even numbers do occur. If $k$ is even, $k>n/2$ then $i=k/2>(n-k)/2=j$, and if $k>2[n/4]+1$ is odd then $i=(k-1)/2>(n-k+1)/2=j$.)
		
		To prove multiplicity freeness, let $i\leq k/2$ be fixed, $E=\{2i+1,\ldots,k\}$ and $L=S_{2i}\times VS_{n-2i}=I_W(\eta_E)$. Then $\eta_E$ extends to $\mu_E\in\Irr(L)$ such that $S_{2i}\times S_{n-2i}\subseteq\Ker(\mu_E)$. Of course, $\mu_E$ also extends $\nu_E$. The other constituents of $\eta_E^L$ are $\mu_E\varphi$, such that $\varphi\in\Irr(L)$ and $Ker(\varphi)\supseteq V$, in particular, all the irreducible constituents of $\nu_E^L$ are of this form. By \cite[(6.11)]{Isaacs:1976}, each $(\mu_E\varphi)^W$ is irreducible. Now, $\chi_i^W=(\nu_E^{VQ})^W=\nu_E^W=(\nu_E^L)^W$ and, by Frobenius reciprocity,
		\[(\nu_E^L,\mu_E\varphi)=(\nu_E,(\mu_E\varphi)_M)=(\nu_E,\nu_E\varphi_M)=(1_M,\varphi_M)=(1_M^L,\varphi)\leq 1,
		\]
		because $M\leq L$ is a Gelfand subgroup. (Consider $M/VS_{2i}\cong S_{k-2i}\times S_{n-k}\leq S_{n-2i}\cong L/VS_{2i}$.) That is, $\nu_E^L$ is multiplicity free. As each of its irreducible constituents are of form $\mu_E\varphi$ and these induce irreducibly to  $W$, so we also get that $\chi_i^W=\nu_E^W$ is multiplicity free, as required.
		
		Given $i\ne j$ the characters $\chi_i^W$ and $\chi_j^W$ share no common constituent if the underlying $\eta_{E_i}$ and $\eta_{E_j}$ are not $W$-conjugate for $E_i=\{2i+1,\ldots,k\},\,E_j=\{2j+1,\ldots,k\}$. Suppose $\sigma$ is such that $\eta_{E_i}^\sigma=\eta_{E_j}$. If $\sigma n=n$ then \[\{\sigma(1),\ldots,\sigma(2i)\}=\{f\mid v_f\notin\Ker(\eta_{E_i}^\sigma) \}=\{f\mid v_f\notin\Ker(\eta_{E_j}) \}=\{1,\ldots,2j \}\]
		have the same cardinality, so $i=j$. If $\sigma m=n$, $2i<m<n$ then $v_{\sigma n}=\sigma v_m\in\Ker(\eta_{E_i}^\sigma)
		$, so
		\[\{\sigma(1),\ldots,\sigma(2i)\}=\{f\mid v_f\notin\Ker(\eta_{E_i}^\sigma) \}=\{f\mid v_f\notin\Ker(\eta_{E_j}) \}=\{1,\ldots,2j \},\]
		again and $i=j$. Finally, if $\sigma m=n$, $m\leq 2i$ then $v_{\sigma n}=\sigma v_m\notin\Ker(\eta_{E_i}^\sigma)$, so
		\[\{\sigma(n),\sigma(2i+1),\ldots
		\sigma(n-1)\}=\{f\mid v_f\notin\Ker(\eta_{E_i}^\sigma) \}=\{f\mid v_f\notin\Ker(\eta_{E_j}) \}=\{1,\ldots,2j \}\]
		have the same cardinality $n-2i=2j$, that is $n=2i+2j$, as required. The claim is established.
		
		Let now $n$ be even, $j=[k/2]>i=n/2-j>0$ and $E_i=\{2i+1,\ldots,k\},\,E_j=\{2j+1,\ldots,k\}$. Let $\sigma$ be the transposition $(1,2j+1)(2,2j+2)\cdots (2i,n)$, hence $\eta_{E_j}^\sigma=\eta_{E_i}$. If $M_i=I_{VQ}(\eta_{E_i})$ and $M_j=I_{VQ}(\eta_{E_j})$ then $\nu_{E_i}\downarrow_{M_j^\sigma\cap M}=\nu_{E_j}^\sigma\downarrow_{M_j^\sigma\cap M}$. By using Mackey's theorem\cite[Problem~(5.6)]{Isaacs:1976},  we get
		\[(\chi_i^W,\chi_j^W)=(\nu_{E_i}^W,\nu_{E_j}^W)=(\nu_{E_i},\nu_{E_j}^W\downarrow_{M_i})=\sum_{W=\cup M_j g M_i}(\nu_{E_i},(\nu_{E_j}^g\downarrow_{M_j^g\cap M_i})^{M_i}).
		\]
		Among the summands is \[(\nu_{E_i},(\nu_{E_j}^\sigma\downarrow_{M_j^\sigma\cap M_i})^{M_i})=(\nu_{E_i},(\nu_{E_i}\downarrow_{M_j^\sigma\cap M_i})^{M_i})=(\nu_{E_i}\downarrow_{M_j^\sigma\cap {M_i}},\nu_{E_i}\downarrow_{M_j^\sigma\cap {M_i}})=1,\]
		hence $1_{Q}^W=\sum_{i=0}^{[k/2]} \chi_i^W$ is not multiplicity free.
	\end{proof}

\end{document}